\documentclass[11pt,a4paper]{amsart}
\usepackage{graphicx} % Required for inserting images
\usepackage{mathrsfs}
\usepackage{amsmath}
\usepackage{amsfonts}
\usepackage{amssymb}
\usepackage{amsthm}
\usepackage{enumitem}
\usepackage{graphicx}
\usepackage{url}
\usepackage{hyperref}
\usepackage{xcolor}
\usepackage{mathtools}
\usepackage[latin1]{inputenc}

\newtheorem{theorem}{Theorem}[section]
\newtheorem{definition}[theorem]{Definition}
\newtheorem{remark}[theorem]{Remark}

\title[Quantum Sobolev spaces of Hilbert-Schmidt operators]{On quantum Sobolev spaces  consisting of Hilbert-Schmidt operators}
\author{Anat\'e K. Lakmon and Yaogan Mensah}

\address{Department of Mathematics, University of Lom\'e, Togo \\1 P.O.Box 1515 Lom\'e , Togo} 
\email{\textcolor[rgb]{0.00,0.00,0.84}{davidlakmon@gmail.com, ORCID :0000-0002-4261-2107 }}

\address{Department of Mathematics, University of Lom\'e, Togo \\1 P.O.Box 1515 Lom\'e , Togo} 
\email{\textcolor[rgb]{0.00,0.00,0.84}{mensahyaogan2@gmail.com,  ORCID : 0000-0002-0028-5515}}

\begin{document}
\maketitle
%\hbox{\it This paper is dedicated to the memory of Professor K.K. Siggini}
\begin{abstract}
In the present paper, we study quantum Sobolev spaces whose elements are operators of the Hilbert-Schmidt class. We construct these Sobolev spaces from the Fourier transform for  operators. Next, we obtain continuous embedding theorems. Finally, we delve into solving partial differential equations where the unknown is an operator. The results  have potential applications in quantum physics, providing a new theoretical basis for relevant research.   
\end{abstract}
{\small Keywords:  Hilbert-Schmidt operator, quantum Fourier transform, locally compact abelian group,  quantum Sobolev space.
\newline
2020 Mathematics Subject Classification: 47B90,   43A25, 47A62, 46N50.}

\section{Introduction}
We can construct Sobolev spaces from a spectral foundation, relying, for example, on the classical or generalized Fourier transform or any other integral transform of the same type that possesses the essential features of the Fourier transform. These include, among others, the Sobolev spaces of Bessel potential type. The classical model that can be cited includes the Sobolev spaces $H^s(\mathbb{R}^n)$ consisting of tempered distributions $u$ such that $(1+|\xi|^2)^{\frac{s}{2}}\widehat{u}\in L^2(\mathbb{R}^n)$,  where $\widehat{u}$ represents the Fourier transform of $u$.  The generalization to topological-algebraic structures in the context of harmonic analysis has seen several extensions. These generalizations include Sobolev spaces on locally compact topological groups, Sobolev spaces on locally compact hypergroups, Sobolev spaces on homogeneous spaces, etc. We emphasize that in this framework weak differentiation is not taken into account, although the results obtained can be invested in solving partial differential equations.
Work already carried out in this context on (locally compact) groups includes the references \cite{Gorka1, Gorka2, Mensah2, Kumar,Mensah4}. As for hypergroups associated with Gelfand pairs, we refer \cite{Bataka1,  Saha}. 

Quantum Sobolev spaces on phase space have been investigated. In \cite{Lafleche},  Lafleche investigated the quantum analogue of the classical Sobolev inequalities in the phase space, with the quantum Sobolev norms defined in terms of Schatten norms of commutators.  He also introduced quantum Besov spaces. 
In \cite{Fulsche2}, Fulsche and  van Luijk gave a new proof of Cordes' characterization of Heisenberg-smooth operators  by using the phase-space formalism of Quantum harmonic analysis and the quantumSobolev spaces introduced in \cite{Lafleche}; they derived Schatten versions of the Calder\'on-Vaillancourt theorem. 
Fulsche and  Galke \cite{Fulsche}  extended  the elements of quantum harmonic analysis as introduced by Werner \cite{Werner} to abelian phase space by which they mean a locally compact abelian group endowed with a Heisenberg multiplier. They then get a joint harmonic analysis of functions and operators. Equipped with such a tool, it appears interesting to take the path that will lead to the construction of Sobolev-type spaces whose elements are operators on a Hilbert space, thus generalizing the results that are known in the context of functions in Lebesgue spaces.

The aim of this paper is to examine a more general aspect of the Sobolev spaces in quantum setting.  We will rely on Schatten-von Neumann spaces instead of classical Lebesgue spaces.  In particular, the Sobolev spaces that will be constructed, and which can modestly be named quantum Sobolev spaces, will be composed of operators of the Hilbert-Schmidt class. We provide two applications of the obtained model to partial differential equations after expressing the associated Laplacian using the quantum Fourier transform.

The rest of the article is presented as follows.
In Section \ref{Preliminaries}, we gather the necessary data to make the document somewhat self-contained. Section \ref{Quantum Sobolev spaces} deals with quantum Sobolev spaces. Applications are provided in Section \ref{Applications}.

\section{Preliminaries}\label{Preliminaries}
In this section, we recall elements that we will use to facilitate the understanding of the content of the article. These are essentially the operators of the Schatten-von Neumann class and the quantum Fourier transform as developed in  \cite{Fulsche}.
\subsection{Schatten class operators}
To write this subsection, we relied on the book by K. Zhu \cite{Zhu}.
Let $H$ be a separable complex Hilbert space. Denote by $\mathcal{B}(H)$ the space of bounded linear operators on $H$ and by $\mathcal{K}(H)$ the set of compact operators on $H$. It is known that $\mathcal{K}(H)$ is a closed self-adjoint two-sided ideal of $\mathcal{B}(H)$. If $T\in \mathcal{K}(H)$ then there exists orthonormal sets $(e_n)$ and $(\varepsilon_n)$ in $H$ such that 
$$Tx=\sum\limits_n\lambda_n\langle x,e_n\rangle \varepsilon_n, \, x\in H, $$
where the non-increasingly arranged sequence $(\lambda_n)_n$   are the singular values of $T$,  that is, the $\lambda_n$  are  the eigenvalues of the operator $|T|:=\sqrt{T^*T}$. Let $0< p<\infty$. The Schatten $p$-class of $H$, denoted $S_p(H)$ is defined as the set of compact operators on $H$ such that the sequence of its singular values $\{\lambda_1,\lambda_2,\cdots\}$ is $p$-summable. For $1\leq p<\infty$, the space $S_p(H)$ is a Banach space  when we equip it with the norm given by 
\begin{equation}
\|T\|_p=\left(\sum\limits_{n=1}^\infty|\lambda_n|^p\right)^{\frac{1}{p}}.
\end{equation}
In particular,  $S_1(H)$ is the space of  trace class operators and $S_2(H)$ the class of Hilbert-Schmidt operators on $H$. By convention, $S_\infty (H)=\mathcal{B}(H)$.

\subsection{Quantum Fourier transform}
We briefly recall the quantum Fourier transform and its properties as described in  \cite{Fulsche}. Throughout the paper, we consider a locally compact abelian group (LCA group) $G$ that is also assumed to be Hausdorff. The Pontrjagin dual of $G$ is denoted $\widehat{G}$. We recall that in this case, $\widehat{G}$ is also a locally compact abelian group and $\widehat{\widehat{G}}\simeq G$.
Let $H$ be a (separable) complex Hilbert space. Let $U(H)$ denote the group of unitary operators on $H$. The first notion that comes into play is that of projective representation. (However, we should note that this notion is not related to the commutativity of the group.)

\begin{definition}
   A unitary projective representation of $G$ over a Hilbert space $H$ is a map $\pi : G\rightarrow U(H)$ which satisfies the relation 
 \begin{equation}   
\forall x,y\in G, \pi(x)\pi(y)=m(x,y) \pi(xy),
\end{equation}
where  the function $m: G\times G \rightarrow \mathbb{T}$ {\rm (the unit circle)} is called  the multiplier (or cocycle) of the projective representation $\pi$. 
\end{definition}
Let us mention that the projective representation considered in this paper is assumed to be irreducible. 
 Throughout the paper, we will write $U_x$ instead of  $\pi(x)$.
 
By the associativity of $G$, the multiplier $m$ has the following property :  
$$\forall x,y,z\in G, \, m(x,y)m(xy,z)=m(x, yz)m(y,z).$$
Let $e$ denote the neutral element of $G$. One may assume that $U_e=I$, the identity operator on $H$. Therefore, $$\pi(x)^*=\overline{m(x,x^{-1})}\pi(x^{-1})$$
where $\pi(x)^*$ is the adjoint of the operator $\pi(x)$.
Set $\sigma(x,y)=\displaystyle\frac{m(x,y)}{m(y,x)}$. If the map $x\longmapsto \sigma(x,\cdot)$, which is a continuous morphism from $G$ into $\widehat{G}$, turns out to be a topological isomorphism, then $m$ is called a Heisenberg multiplier (and  $\sigma$ is  called a symplectic bicharacter of $G$ \cite{Digernes}). Obviously,
$$\forall x,y\in G, \,U_xU_y=\sigma(x,y)U_yU_x.$$
Important assumptions that allowed the theory in \cite{Fulsche} to work are :  
\begin{enumerate}
    \item the square integrability of the projective representation $\pi$,  that is, there exists non-zero vectors $\phi,\psi\in H$ such that the mapping $x\mapsto \langle \pi (x)\phi,\psi\rangle$ belongs to $L^2(G)$, 
    \item the multiplier $m$ is a Heisenberg multiplier, 
    and  satisfies : 
    $$m(x,y)=m(x^{-1},y^{-1}),\, \forall x,y\in G.$$ 
\end{enumerate}
Sometimes, it will be assumed that the projective representation is integrable, that is, there exists a non-zero vector $\psi \in H$ such that the mapping $x\mapsto\langle \pi (x)\psi,\psi\rangle$ belongs to $L^1(G)$, 

When dealing with a Heisenberg multiplier, one can identify $G$ and $\widehat{G}$ through the bicharacter $\sigma$. 
We recall the definition of the quantum Fourier transform (the Fourier transform of an operator) and its properties that one has the right to expect. 
\begin{definition}\cite[Definition 6.12]{Fulsche}
For  $T\in S_1(H)$, the Fourier transform of $T$
is given by
\begin{equation}
   \mathcal{F}_U(T)(\xi)=tr(TU_\xi^\ast),\,\xi \in \widehat{G}. 
\end{equation}
\end{definition}
Here, $tr(A)$ is the trace of the operator $A$.  
If $T\in S_1(H)$ is such that $\mathcal{F}_U(T)\in L^1(\widehat{G})$, then the reconstruction formula is given by
\begin{equation}
    T=\int_{\widehat{G}}\mathcal{F}_U(T)(\xi)U_\xi d\xi.
\end{equation}

\begin{theorem}(Riemann-Lebesgue lemma)\cite[Proposition 6.16]{Fulsche}
If $T\in  S_1(H)$. Then,  $\mathcal{F}_U(T)\in C_0(\widehat{G})$ and $\|\mathcal{F}_U(T)\|_\infty\leq \|T\|_{S_1(H)}$, where $\|\mathcal{F}_U(T)\|_\infty=\sup\limits_{\xi\in \widehat{G}}|\mathcal{F}_U(T)(\xi)|$. 
\end{theorem}

We have the Plancherel's theorem for operators : 
\begin{theorem}\cite[Theorem 6.23]{Fulsche}\label{Plancherel}
Assume that the representation is integrable.
Then, the Fourier transform extends to a unitary operator $\mathcal{F}_U : S_2(H)\rightarrow L^2(\widehat{G})$ with adjoint  $\mathcal{F}_U^{-1}:  L^2(\widehat{G})\rightarrow  S_2(H)$.   
\end{theorem}

\begin{theorem}\cite[Proposition  6.25]{Fulsche}\label{Hausdorff-Young}
Assume that the representation is integrable.
Let  $p\in [1, 2]$  and let $q$ be such that $\dfrac{1}{p}+\dfrac{1}{q}=1$. Then,
$$\|\mathcal{F}_U(T)\|_{L^q(\widehat{G})}\leq \|T\|_{S_p(H)},\, T\in S_p(H).$$
\end{theorem}
Hereafter is the Hausdorff-Young inequality for $\mathcal{F}_U^{-1}$. 
\begin{theorem}\cite[Proposition  6.26]{Fulsche}\label{Hausdorff-Young reverse}
Assume that the representation is integrable.
Let  $p\in [1, 2]$  and let $q$ be such that $\dfrac{1}{p}+\dfrac{1}{q}=1$. Then,
$$\|\mathcal{F}_U^{-1}(f)\|_{S_q(H)}\leq \|f\|_{L^p(\widehat{G})},\, f\in L^p(\widehat{G}).$$
\end{theorem}
\section{Quantum Sobolev spaces : continuous embeddings} \label{Quantum Sobolev spaces}
This section contains our main results. We introduce Sobolev spaces consisting of Hilbert-Schmidt operators. We then prove continuous embedding results.

Let $\gamma : \widehat{G} \rightarrow (0,\infty) $ be a mesurable function. Let $s$ be a positive real number. 
\begin{definition}
We call  (quantum) Sobolev space the following set:
\begin{equation}
\mathfrak{H}^s_\gamma(G,H)=\left\{T\in S_2(H) : \int_{\widehat{G}}(1+\gamma(\xi)^2)^s|\mathcal{F}_U(T)(\xi)|^2d\xi<\infty\right\}
\end{equation}
equipped with the norm 

\begin{equation}\label{norm H}
  \|T\|_{\mathfrak{H}^s_\gamma(G,H)}=\left(\int_{\widehat{G}}(1+\gamma(\xi)^2)^s|\mathcal{F}_U(T)(\xi)|^2d\xi\right)^{\frac{1}{2}}.  
\end{equation}
\end{definition}
We see that $\mathfrak{H}^s_\gamma(G,H)$ is realized as a space of operators in contrast to the classical cases where Sobolev spaces are spaces of functions.  

\begin{remark}{\rm
In the sequel, if we write $X\hookrightarrow Y$, it will mean that the Banach space $X$  embeds continuously into the Banach space $Y$, that is,   $X\subset Y$ and there exists a positive constant $C$ such that $\|x\|_Y\leq C\|x\|_X,\,\forall x\in X$.
}\end{remark}
\begin{theorem}
 The Sobolev space $\mathfrak{H}^s_\gamma(G,H)$  is a Hilbert space and  its inner product is given by
 $$\langle S, T\rangle= \int_{\widehat{G}}(1+\gamma(\xi)^2)^s\mathcal{F}_U(S)(\xi)\overline{\mathcal{F}_U(T)(\xi)}d\xi,\, S,T\in  \mathfrak{H}^s_\gamma(G,H).$$
\end{theorem}
\begin{proof}
  It is sufficient to observe that the Sobolev space $\mathfrak{H}^s_\gamma(G,H)$ is isometrically isomorphic to the Hilbert space $L^2(\widehat{G})$ via the  mapping $T\mapsto (1+\gamma(\cdot)^2)^{\frac{s}{2}}\mathcal{F}_U(T)$. Moreover, the expression of the inner product is deduced from the norm in $\mathfrak{H}^s_\gamma(G,H)$.
\end{proof}

The following theorem shows that the Sobolev spaces  can be ordered according to "smoothness" parameter.
\begin{theorem}
 If $s_2>s_1>0$,  then  $\mathfrak{H}^{s_2}_\gamma(G,H)\hookrightarrow \mathfrak{H}^{s_1}_\gamma(G,H)$ and
 $$ \|T\|_{\mathfrak{H}^{s_1}_\gamma(G,H)}\leq  \|T\|_{\mathfrak{H}^{s_2}_\gamma(G,H)}, \, \forall\, T\in\mathfrak{H}^{s_2}_\gamma(G,H).$$
\end{theorem}
\begin{proof}
    The result is due to the fact that $(1+\gamma(\xi)^2)^{s_1}\leq (1+\gamma(\xi)^2)^{s_2}$.
\end{proof}
The Sobolev space $\mathfrak{H}^s_\gamma(G,H)$ embeds continuously in the space of Hilbert-Schmidt operators. 
\begin{theorem}
    We have $$\mathfrak{H}^s_\gamma(G,H)\hookrightarrow S_2(H).$$
    Moreover, $\forall\, T\in \mathfrak{H}^s_\gamma(G,H), \|T\|_{S_2(H)}\leq  \|T\|_{\mathfrak{H}^s_\gamma(G,H)}$.
\end{theorem}
\begin{proof}
 By Theorem \ref{Plancherel}, we have $\|T\|_{S_2(H)}=\|\mathcal{F}_U(T)\|_{ L^2(\widehat{G})}$. Therefore,
 \begin{align*}
   \|T\|_{S_2(H)}&=  \left(\int_{\widehat{G}}|\mathcal{F}_U(T)(\xi)|^2d\xi\right)^{\frac{1}{2}}\\
   &\leq \left(\int_{\widehat{G}}(1+\gamma(\xi)^2)^s|\mathcal{F}_U(T)(\xi)|^2d\xi\right)^{\frac{1}{2}}=\|T\|_{\mathfrak{H}^s_\gamma(G,H)}.
 \end{align*}
\end{proof}
The following theorem gives a sufficient condition for the continuous embedding of the Sobolev space $\mathfrak{H}^s_\gamma(G,H)$ into $\mathcal{B}(H)$, the space of bounded operators on $H$.

\begin{theorem}
 If  $(1+\gamma(\cdot)^2)^{-\frac{s}{2}}\in L^2(\widehat{G})$, then $\mathfrak{H}^s_\gamma(G,H)\hookrightarrow\mathcal{B}(H)$ and $\forall\, T\in \mathfrak{H}^s_\gamma(G,H)$, 
 $$\|T\|_{op}\leq \|(1+\gamma(\cdot)^2)^{-\frac{s}{2}}\|_{L^2(\widehat{G})}\|T\|_{\mathfrak{H}^s_\gamma(G,H)}.$$
\end{theorem}
 \begin{proof}
Let $T\in \mathfrak{H}^s_\gamma(G,H)$. In order to apply the Fourier inversion formula, we will first show that $\mathcal{F}_U(T)$ is integrable. 
\begin{align*}
\int_{\widehat{G}}|\mathcal{F}_U(T)(\xi)|d\xi&=\int_{\widehat{G}}(1+\gamma (\xi)^2)^{\frac{s}{2}}|\mathcal{F}_U(T)(\xi)|(1+\gamma (\xi)^2)^{-\frac{s}{2}}d\xi\\
&\leq \left(\int_{\widehat{G}}(1+\gamma (\xi)^2)^s|\mathcal{F}_U(T)(\xi)|^2d\xi\right)^{\frac{1}{2}}\left(\int_{\widehat{G}}(1+\gamma (\xi)^2)^{-s}d\xi\right)^{\frac{1}{2}}\\
& \mbox{(by the Cauchy-Schwarz inequality.)}\\
&\leq \|T\|_{\mathfrak{H}^s_\gamma(G,H)}\|1+\gamma (\cdot)^2)^{-\frac{s}{2}}\|_{L^2(\widehat{G})}<\infty.
\end{align*}
Thus,  $\mathcal{F}_U(T)$ is integrable.

Now, let $\eta$ be a unit vector in $H$. Using the Fourier inversion formula, we have 
 \begin{align*}
     \|T\eta\|_H&=\left\|\int_{\widehat{G}}\mathcal{F}_U(T)(\xi)U_\xi\eta d\xi\right\|_H\\
     &\leq \int_{\widehat{G}}\left\|(\mathcal{F}_U(T)(\xi)U_\xi\eta \right\|_H d\xi\\
     &\leq \int_{\widehat{G}}|\mathcal{F}_U(T)(\xi)|\left\|U_\xi \right\|_{op}  \|\eta\|_H d\xi  \\
     &=\int_{\widehat{G}}|\mathcal{F}_U(T)(\xi)|d\xi\\
& \mbox{(because $U_\xi$ is unitary and $\eta$ is a unit vector.)}\\
&\leq \|T\|_{\mathfrak{H}^s_\gamma(G,H)}\|1+\gamma (\cdot)^2)^{-\frac{s}{2}}\|_{L^2(\widehat{G})}\\
 \end{align*}
 Therefore, $\|T\|_{op}:=\sup\limits_{\|\eta\|=1}\|T\eta\|\leq \|T\|_{\mathfrak{H}^s_\gamma(G,H)}\|(1+\gamma(\cdot)^2)^{-\frac{s}{2}}\|_{L^2(\widehat{G})}.$
\end{proof}
The following theorem provides a sufficient condition for the continuous embedding of the Sobolev space $\mathfrak{H}^s_\gamma(G,H)$ into some Schatten classes of operators. 
\begin{theorem}
Let $\alpha>s$. Set $\alpha^*=\frac{2\alpha}{\alpha-s}$. If $\dfrac{1}{1+\gamma(\cdot)^2}\in L^\alpha (\widehat{G})$, then $\mathfrak{H}^s_\gamma(G,H)\hookrightarrow S_{\alpha^*}(H)$. Moreover, 
$$\|T\|_{S_{\alpha^*}(H)}\leq \|T\|_{\mathfrak{H}^s_\gamma(G,H)}\left\|\dfrac{1}{1+\gamma(\cdot)^2}\right\|_{L^\alpha (\widehat{G})}^{\frac{s}{2}}, \forall T\in \mathfrak{H}^s_\gamma(G,H).$$
\end{theorem}
\begin{proof}
Let $T\in \mathfrak{H}^s_\gamma(G,H)$.  Let $p=\dfrac{2\alpha}{\alpha+s}$. Then $1<p<2$ and $\dfrac{1}{p}+\dfrac{1}{\alpha^*}=1$. Then, by the Hausdorff-Young inequality for $\mathcal{F}_U^{-1}$ (Theorem \ref{Hausdorff-Young reverse}), we get 
$$\|T\|_{S_{\alpha^*}(H)}\leq \|\mathcal{F}_U(T)\|_{L^p(\widehat{G})}.$$
Moreover, 
\begin{align*}
\|\mathcal{F}_U(T)\|_{L^p(\widehat{G})}^p&=\int_{\widehat{G}}|\mathcal{F}_U(T)(\xi)|^p d\xi\\
    &= \int_{\widehat{G}}|\mathcal{F}_U(T)(\xi)|^p \dfrac{(1+\gamma(\xi)^2)^{\frac{sp}{2}}}{(1+\gamma(\xi)^2)^{\frac{sp}{2}}}d\xi\\
    &=\int_{\widehat{G}}(1+\gamma(\xi)^2)^{\frac{sp}{2}}|\mathcal{F}_U(T)(\xi)|^p \dfrac{1}{(1+\gamma(\xi)^2)^{\frac{sp}{2}}}d\xi.
\end{align*}
Observing that $\dfrac{p}{2}+\dfrac{2-p}{2}=1$, we apply the H\"older inequality.  We obtain
\begin{align*}
\|\mathcal{F}_U(T)\|_{L^p(\widehat{G})}^p & \leq \left(\int_{\widehat{G}}(1+\gamma(\xi)^2)^{s}|\mathcal{F}_U(T)(\xi)|^{2} \right)^{\frac{p}{2}}\left( \int_{\widehat{G}}\dfrac{d\xi}{(1+\gamma(\xi)^2)^{\frac{sp}{2-p}}}\right)^{\frac{2-p}{2}}.
\end{align*}
This leads to 
\begin{align*}
\|\mathcal{F}_U(T)\|_{L^p(\widehat{G})} & \leq \left(\int_{\widehat{G}}(1+\gamma(\xi)^2)^{s}|\mathcal{F}_U(T)(\xi)|^{2} \right)^{\frac{1}{2}}\left( \int_{\widehat{G}}\dfrac{d\xi}{(1+\gamma(\xi)^2)^\alpha}\right)^{\frac{2-p}{2p}}\\
&=\|T\|_{\mathfrak{H}^s_\gamma(G,H)}\left\|\dfrac{1}{(1+\gamma(\cdot)^2)}\right\|_{L^\alpha (\widehat{G})}^{\frac{s}{2}}.
\end{align*}
Finally $$\|T\|_{S_{\alpha^*}(H)}\leq \|T\|_{\mathfrak{H}^s_\gamma(G,H)}\left\|\dfrac{1}{1+\gamma(\cdot)^2}\right\|_{L^\alpha (\widehat{G})}^{\frac{s}{2}}.$$
\end{proof}

\section{Applications}\label{Applications}
In this section, we focus on applications. More precisely, we solve two partial differential equations where the unknowns are in the Hilbert-Schmidt class. Throughout this section, we will call {\it transformer} any application whose input and output are operators.
\subsection{The quantum Poisson-type equation}\label{Poisson}
We consider the quantum Poisson-type equation. 
\begin{equation}\label{Poisson-type equation}
    (\mathscr{I}-\Delta )T=S
\end{equation}
where $\Delta$ is the Laplacian, $\mathscr{I}$ is the identity transformer, $S$ is a given operator on a Hilbert space $H$ and $T$ is the unknown. The transformer associated with Equation (\ref{Poisson-type equation}) is $\mathscr{L}_0=\mathscr{I}-\Delta$ where so that $\mathscr{L}_0 T=S.$
Our goal is to make sense of this equation by expressing  the transformer $\mathscr{L}_0$ in Fourier space by the quantum Fourier transformation $\mathcal{F}_U$.
Let us start with the following heuristic approach.
\subsubsection{Heuristic approach}
The Fourier transform of a Schwartz function $f$ on $\mathbb{R}$ is given by 
$$[\mathscr{F}f](\nu)=\int_{\mathbb{R}}f(x)e^{-2i\pi\nu x}dx.$$
The one-dimensional Laplacian is $\displaystyle\frac{d^2}{dx^2}$ and  we know that:   $$\left[\mathscr{F}\left(\frac{d^2f}{dx^2}\right)\right](\nu)=-(2\pi\nu)^2[\mathscr{F}f](\nu).$$
We obtain $$\frac{d^2f}{dx^2}=-\mathscr{F}^{-1}[(2\pi\nu)^2\mathscr{F}f].$$
Therefore, 
\begin{align*}
 f- \frac{d^2f}{dx^2}= \mathscr{F}^{-1}\mathscr{F}f+ \mathscr{F}^{-1}[(2\pi\nu)^2\mathscr{F}f]\\
  =\mathscr{F}^{-1}[(1+(2\pi\nu)^2)\mathscr{F}f].
\end{align*}
 Thus, 
\begin{equation}\label{Heuristic1}
\left(I-\frac{d^2}{dx^2}\right)f=\mathscr{F}^{-1}[(1+(2\pi\nu)^2)\mathscr{F}f],  
\end{equation}
where $I$ is the identity operator. 

\subsubsection{Quantum setting}
Following the equality (\ref{Heuristic1}), we consider that in the quantum setting the transformer $\mathscr{L}_0$ is expressed as follows:
\begin{equation}
   \mathscr{L}_0 T=\mathcal{F}_U^{-1}\left[\left(1+\gamma(\xi)^2\right)\mathcal{F}_U(T)\right]. 
\end{equation}
Therefore, the domain of the transformer $\mathscr{L}_0$ is  
$$D(\mathscr{L}_0)=\mathfrak{H}_\gamma^2(G,H)=\left\{ T\in S_2(H) : \int_{\widehat{G}}(1+\gamma(\xi)^2)^2|\mathcal{F}_U(T)(\xi)|^2d\xi<\infty\right\}.$$
Let us recall from \ref{norm H} that the norm on $\mathfrak{H}_\gamma^2(G,H)$ is given by
$$
 \|T\|_{\mathfrak{H}^2_\gamma(G,H)}=\left(\int_{\widehat{G}}(1+\gamma(\xi)^2)^2|\mathcal{F}_U(T)(\xi)|^2d\xi\right)^{\frac{1}{2}}. $$ 
\begin{theorem}
If $S\in S_2(H)$, then the quantum Poisson-type equation $\mathscr{L}_0 T=S$ has a unique solution $T_0$ in $D(\mathscr{L}_0)$. Moreover, 
$$\|T_0\|_{\mathfrak{H}_\gamma^2(G,H)}=\|S\|_{S_2(H)}.$$
\end{theorem}
\begin{proof}
Let $T\in D(\mathscr{L}_0)$.
 \begin{align*}
  \mathscr{L}_0 T=S &  \Longleftrightarrow   \mathcal{F}_U^{-1}[(1+\gamma(\xi)^2)\mathcal{F}_U(T)]=S\\
  & \Longleftrightarrow T=\mathcal{F}_U^{-1}\left(\displaystyle\frac{\mathcal{F}_U(S)}{1+\gamma(\xi)^2}\right).
  \end{align*}  
  Thus, the unique solution is $T_0=\mathcal{F}_U^{-1}\left(\displaystyle\frac{\mathcal{F}_U(S)}{1+\gamma(\xi)^2}\right)$. 
  
  Moreover,
  \begin{align*}
   \|T_0\|_{\mathfrak{H}_\gamma^2(G,H)}&= \left(\int_{\widehat{G}}(1+\gamma(\xi)^2)^2|\mathcal{F}_U(T_0)(\xi)|^2d\xi\right)^{\frac{1}{2}}\\
   &=\left(\int_{\widehat{G}}|\mathcal{F}_U(S)(\xi)|^2d\xi\right)^{\frac{1}{2}}\\
   &=\|\mathcal{F}_U(S)\|_{L^2(\widehat{G})}\\
   &=\|S\|_{S_2(H)}\\
   &(\mbox{by the Plancherel theorem}).
  \end{align*}
  
\end{proof}

\subsection{Quantum generalized bosonic string equation}
With the same notation as in Subsection \ref{Poisson}, we consider the quantum generalized bosonic string equation :
\begin{equation}
    (\Delta e^{c\Delta}-\mathscr{I} )T=S. 
\end{equation}
where $c$ is a positive constant. 
The associated transformer is expressed as
$$\mathscr{L}_c=\Delta e^{c\Delta}-\mathscr{I}.$$
Note that $\mathscr{L}_0$ is obtained from $\mathscr{L}_c$ by allowing $c$ to take the value 0, this justifies their notation.
\subsubsection{Heuristic approach}
For the heuristic approach (see \cite{Gorka2}), we consider the function $u(x)=xe^{-cx}$ and we write the formal Taylor series of $u$ : 
$$u(x)=\sum_{k=0}^\infty\frac{u^{k}(0)}{k!}x^k.$$
Then, replacing $x$ by the one dimensional Laplacian $\displaystyle\frac{d^2}{dx^2}$, we get 
\begin{align*}
    u\left(\frac{d^2}{dx^2}\right)&=\sum_{k=0}^\infty\frac{u^{k}(0)}{k!}\left(\frac{d^2}{dx^2}\right)^k\\
    \Rightarrow u\left(\frac{d^2}{dx^2}\right)f&=\sum_{k=0}^\infty\frac{u^{k}(0)}{k!}\left(\frac{d^2}{dx^2}\right)^kf\\
  \Rightarrow   \mathscr{F}\left( u\left(\frac{d^2}{dx^2}\right)f\right)&=\sum_{k=0}^\infty\frac{u^{k}(0)}{k!}\mathscr{F}\left(\left(\frac{d^2}{dx^2}\right)^kf\right)\\
    &=\sum_{k=0}^\infty\frac{u^{k}(0)}{k!}(2i\pi\nu)^{2k}\mathscr{F}f\\
    &=u(-(2\pi\nu)^2)\mathscr{F}f.
\end{align*}
Therefore, $$ u\left(\frac{d^2}{dx^2}\right)f=\mathscr{F}^{-1}\left(u\left(-(2\pi\nu)^2\right)\mathscr{F}f\right).$$

\begin{align*}
 \Rightarrow   u\left(\frac{d^2}{dx^2}\right)f-f&= \mathscr{F}^{-1}(u(-(2\pi\nu)^2)\mathscr{F}f)- \mathscr{F}^{-1}\mathscr{F}f.
\end{align*}
Finally, 
\begin{equation}\label{Bosonic heuristic}
  \left (\frac{d^2}{dx^2}e^{-c\frac{d^2}{dx^2}}-I\right)f=-\mathscr{F}^{-1}[((2\pi\nu)^2e^{c(2\pi\nu)^2}+1)\mathscr{F}f]. 
\end{equation}
\subsubsection{Quantum setting}
Returning to our framework and  mimicking (\ref{Bosonic heuristic}), we define the transformer $\mathscr{L}_c$ as follows :
\begin{equation}
    \mathscr{L}_cT=-\mathcal{F}_U^{-1}\left[\left(1+\gamma(\xi)^2e^{c\gamma(\xi)^2}\right)\mathcal{F}_U(T)\right],
\end{equation}
and the domain of $\mathscr{L}_c$ is  

\begin{equation}
 D(\mathscr{L}_c)=\mathfrak{H}^{\infty}_{c,\gamma}(G,H) :=\left\{T\in S_2(H) :  \int_{\widehat{G}}(1+\gamma(\xi)^2e^{c\gamma(\xi)^2})^2|\mathcal{F}_U(T)(\xi)|^2d\xi<\infty  \right\} . 
\end{equation}
We provide the space $\mathfrak{H}^{\infty}_{c,\gamma}(G,H)$ with the following norm :
\begin{equation}
    \|T\|_{\mathfrak{H}^{\infty}_{c,\gamma}(G,H)}=\left(   \int_{\widehat{G}}(1+\gamma(\xi)^2e^{c\gamma(\xi)^2})^2|\mathcal{F}_U(T)(\xi)|^2d\xi \right)^{\frac{1}{2}}.
\end{equation}

\begin{theorem}
 If $S\in S_2(H)$, then  the  quantum generalized bosonic string equation  $\mathscr{L}_cT=S$ has a unique solution $T_0$ and 
 $$\|T_0\|_{\mathfrak{H}^{\infty}_{c,\gamma}(G,H)}=\|S\|_{S_2(H)}.$$
\end{theorem}
\begin{proof}
  It is clear that $T_0=-\mathcal{F}_U^{-1}\left(\displaystyle\frac{\mathcal{F}_U(S)}{1+\gamma(\xi)^2e^{c\gamma(\xi)^2}}\right)$  is the unique solution. Moreover,  using the norm in $\mathfrak{H}^{\infty}_{c,\gamma}(G,H)$, we get 
  \begin{align*}
    \|T_0\|_{\mathfrak{H}^{\infty}_{c,\gamma}(G,H)}&=\|\mathcal{F}_U(S)\|_{L^2(\widehat{G})}.   
  \end{align*}
  Finally, we achieve the proof by  the support of the Plancerel theorem (Theorem \ref{Plancherel}). 
\end{proof}

\section{Conclusion}
We introduced a family of Sobolev spaces related to operators on a Hilbert space. We studied some of their properties, particularly the continuous embedding relationships that link them with other spaces of operators. We have provided applications to equations where the unknowns are operators. We are convinced that the results outlined here can be useful for applications in Physics.

\subsection*{Competing interests} The authors have no competing interests to declare that are relevant to the content of this article.

\end{document}